\documentclass[12pt]{article}

\usepackage{graphicx}
\usepackage{a4}
\usepackage{amssymb}
\usepackage{amsmath}
\usepackage{amsthm}
\usepackage{comment}

\newtheorem{theorem}{Theorem}[section]
\newtheorem*{theorem*}{Theorem}
\newtheorem{proposition}[theorem]{Proposition}
\newtheorem{lemma}[theorem]{Lemma}
\newtheorem{corollary}[theorem]{Corollary}
\newtheorem{remark}[theorem]{Remark}

\begin{document}

\title{Convergence properties of dynamic mode decomposition for analytic interval maps}

\author{Elliz Akindji$^{1)}$, Julia Slipantschuk$^{2)}$,\\
Oscar F. Bandtlow$^{1)}$,
and Wolfram Just$^{3) \dagger}$\\
$^{1)}$School of Mathematical Sciences,\\ Queen Mary University of London,
London, UK\\
$^{2)}$Department of Mathematics,\\  University of Warwick,
Coventry, UK\\
$^{3)}$Institute of Mathematics,\\ University of Rostock,
Rostock, Germany\\
$^\dagger$E-mail: wolfram.just@uni-rostock.de
}

\date{10th April 2024}

\maketitle

\begin{abstract}
Extended dynamic mode decomposition (EDMD) is a data-driven algorithm
for approximating spectral data of the
Koopman operator associated to a dynamical system,
combining a Galerkin method of order $N$ and collocation method of order $M$.
Spectral convergence of this method subtly depends on appropriate choice of
the space of observables. For
chaotic analytic full branch maps of the interval, we derive a constraint between $M$ and $N$ guaranteeing spectral convergence of EDMD. 
\end{abstract}



\section{Context and Results}\label{sec:1}

Extended dynamic mode decomposition (EDMD) is an increasingly popular tool
for data analysis in complex systems which is used to identify
dynamically relevant modes of a dynamical system based on observations, see for example
\cite{MeBa_PD04, Me_05, RoMeBaScHe_JFM09, Schm_JFM10, BuMoMe_CHAOS12,WiKeRo_JNS15, KBBP_16}.
The rationale of this approach consists in computing effective
modes in dynamical systems, which follows ideas tracing their origins
in the context of statistical data analysis \cite{Karh_AASF47}.
At their core, these methods condense the dynamical observations into a suitably chosen effective
linear evolution matrix. The eigenvalues and eigenvectors of this
matrix then provide information concerning the relevance and structure of the
effective degrees of freedom of the system.
While the idea of EDMD has been predominantly pushed by applications in fluid dynamics,
the method also aims more broadly to provide useful insight into a variety of
real world data analysis problems. Conceptually, it
is based on the observation that general dynamical systems can be described
by global evolution operators, known as transfer or Perron-Frobenius operators, or
their formal adjoints, known as Koopman operators.

In the last decade considerable
progress has been made in the theoretical underpinning of EDMD, for example, by studying
data-driven methods to
approximate Perron-Frobenius and Koopman operators which feature in
the work of Dellnitz and Junge \cite{DeJu_JNA99}
and which have since been extended in many ways, see
\cite{DeFrSe_NONL00, FrGoQu_JCD14, KlKoSc_JCD16, KlGePeSc_NONL18, Gian_ACHA19,DGS_21, ZZ_23} to name but a few. For broader context, a comparison of various data-driven algorithms can be found in
\cite{KlNuKoWuKeScNo_JNS18},
\cite{Gian_ACHA19} gives a comprehensive overview of the historical context,
and \cite{Co_23a} provides a recent review of the many variants of dynamic mode decomposition algorithms.
In particular, substantial progress has been made in the fundamental understanding
of EDMD as a finite-rank approximation scheme of the Koopman operator.
For a dynamical system given by a map $T$, this operator is given by
composition with $T$, and plays a crucial role,
in particular for the study of decay of correlations.
In most of the literature concerned with data-driven approaches,
the operator is considered on the space of square-integrable functions, where for mixing dynamical systems it has a single simple eigenvalue, while other dynamical features are hidden in the continuous spectrum. With data-driven methods it is possible to rigorously approximate the spectral measure in this setting, as for example is done in \cite{KPM_20} using Christoffel-Darboux kernels, or in \cite{CB_24} with a residual-based approach to remove spurious eigenvalues,
see also \cite{Co_23b} for related results.

On the other hand, for certain chaotic dynamical systems of sufficient regularity and with suitably chosen observables, EDMD correctly 
determines rates of correlation decay.
This can be understood by enriching $L^2$ with generalised functions or distributions,
thereby turning the Koopman operator on this extended space into a
(quasi)-compact operator with discrete spectrum well-approximated
by the EDMD scheme, see \cite{SlBaJu_CNSNS20,BJS_23,Wo_23}.
This is explained by the fact that the Koopman operator is adjoint to the transfer operator,
which restricted to a dual space (densely and continuously embedded in $L^2$)
enjoys strong spectral properties, see, for example, \cite{KeLi_99,Bal_00}.
This structure can be seen as an instance of a rigged Hilbert space, see \cite{SlBaJu_CNSNS20} for a discussion of this point of view  or \cite{SATB_96} for a previous use in the dynamical systems context and
the very recent work \cite{IHIK_24} in the data-driven setting.

Nevertheless there remain substantial open questions regarding
convergence and quantitative accuracy of the EDMD method.
In this article we aim to contribute to this challenge.
In order to make quantitative progress,
in the spirit of \cite{SlBaJu_CNSNS20,BJS_23} we focus
on the simplest kind of complex dynamical systems,
chaotic expansive one-dimensional maps,
whose statistical long-term behaviour is well understood.
We consider a discrete dynamical system given by a map on an interval,
say $T \colon [-1,1] \rightarrow [-1,1]$, and the associated
Koopman operator formally given by $(\mathcal{K} f)(x)=f(T(x))$.

To briefly and informally summarize the EDMD algorithm introduced in \cite{WiKeRo_JNS15}, one assumes
that the dynamics (in our case given by the interval map $T$) is observed through a set of
$N$ observables $\{ \psi_0,\ldots, \psi_{N-1}\}$. Further one assumes
that the dynamics is recorded at $M$ nodes $\{x_0,\ldots,x_{M-1}\}$ in the phase space, which might arise from times series data, be sampled from a distribution or be a predefined set. One then solves the generalised
eigenvalue problem given by
\begin{equation}\label{eq:1.1}
\lambda \sum_{n=0}^{N-1} H^{(M)}_{kn} u_n = \sum_{n=0}^{N-1} G^{(M)}_{k n} u_n\, ,
\end{equation}
where
\begin{equation}\label{eq:1.2}
\begin{aligned}
H^{(M)}_{kn} &= \frac{1}{M} \sum_{m=0}^{M-1} \psi_k(x_m) \psi_n(x_m) \\
G^{(M)}_{kn} &= \frac{1}{M} \sum_{m=0}^{M-1} \psi_k(T(x_m)) \psi_n(x_m)
\end{aligned}
\end{equation}
define the $N\times N$ square matrices $\hat H^{(M)}_N$ and $\hat G^{(M)}_N$. Ideally,
the solutions of the problem \eqref{eq:1.1}
are good approximations for the spectral data of the
(appropriately defined) Koopman operator. This, however, depends sensitively on the choice and the number $N$ of observables, and the choice and the number $M$ of nodes.

To emphasise the challenge we face, consider the basic textbook example of a piecewise linear full branch
map of the interval, the skewed doubling map on $[-1,1]$ given by
\begin{equation}\label{eq:3.1}
T_D(x)= \left\{
\begin{array}{rcr}
-1+2 (x+1)/(1+a) & \mbox{ if } & -1 \leq x\leq a\\
1+2 (x-1)/(1-a) & \mbox{ if } & a < x \leq 1
\end{array} \right.,
\end{equation}
where the parameter $a\in (-1,1)$ determines the skew of the map.
For the set of nodes we simply take
a lattice of equidistant points, $x_m=-1+(2m+1)/M$, $m=0,\ldots, M-1$. 
For the set of observables, Fourier modes seem
to be a sensible choice, that is $\psi_k(x)=\exp(i \pi (k-K)x)$ where $N=2 K+1$ is odd and
$k=0,\ldots, N-1$. If $M > N$, the matrix $\hat H^{(M)}_N$ is
diagonal\footnote{Here, for simplicity, we replace the second observable in
\eqref{eq:1.2} by its complex conjugate, which amounts to relabelling the matrix indices.}.
The outcome of the EDMD algorithm, that is, the eigenvalues in \eqref{eq:1.1}
for certain values of $N$ and $M$, are shown in the left panel of Figure~\ref{fig:1.1}.
These eigenvalues are scattered over the complex plane
and the data
seem to display a very slow convergence when increasing the number of nodes and observables.
Furthermore, these data do not seem to correspond to the
rates of correlation decay
of the
skewed doubling map.  For full-branch linear interval maps the (suitably defined) decay rates are in fact well known,
see \cite{MoSoOs_PTP81} for an elementary account or the appendix of \cite{SlBaJu_JPA13} in a slightly more general setting. 
In particular, for the map \eqref{eq:3.1} these decay rates are given
by $\lambda_n=((1+a)/2)^{n+1}+
((1-a)/2)^{n+1}$, $n\in \mathbb{N}_0$. Since EDMD produces counterintuitive outcomes
already in this simple setup there are certainly some issues which need to be clarified.

An alternative setting arises from choosing monomials as the observables,
that is, $\psi_k(x) = x^k$ for
$k = 0, \ldots,  N-1$, while keeping the same set of equidistant points as the nodes. 
The results of applying EDMD in this case are presented in the right panel of
Figure~\ref{fig:1.1}.
We observe that in this case,
the eigenvalues in \eqref{eq:1.1} do approximate $\lambda_n$, however the convergence strongly depends
on the number of chosen nodes $M$, and in particular can deteriorate for larger $N$ if $M$ is not increased appropriately.

\begin{figure}[!h]
\includegraphics[width=0.5 \textwidth]{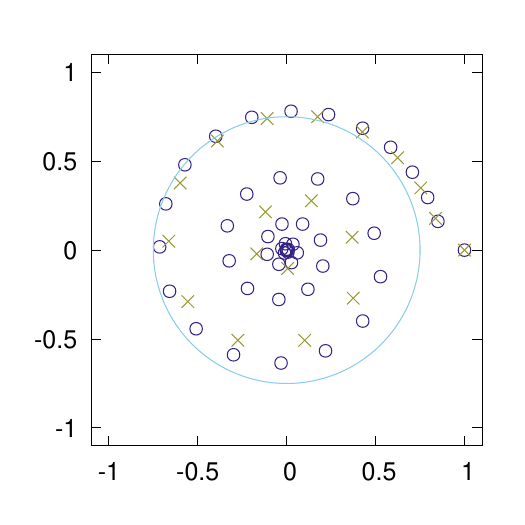}
\includegraphics[width=0.5 \textwidth]{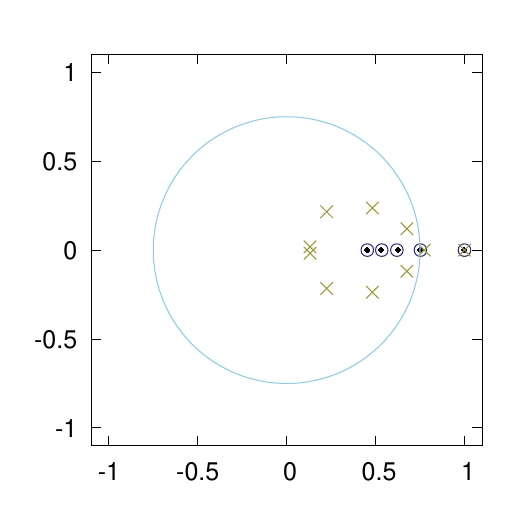}

\caption{Eigenvalues in the complex plane
for the skewed doubling map (Equation (\ref{eq:3.1})) with $a=1/\sqrt{2}$, obtained via EDMD using
Equation (\ref{eq:1.1}). Left: The chosen observables are Fourier modes with
$N=50$, $M=10000$ (blue, circle) and $N=20$, $M=5000$ (amber, cross).
The circle has a radius determined by the correlation decay rate of smooth
observables, $\lambda_1 = (1+a^2)/2$. Right: The chosen observables are monomials with $N=5$, $M=10000$ (blue, circle)
and $N=10$, $M=10000$ (amber, cross). The dots correspond to $\lambda_n=((1+a)/2)^{n+1}+
((1-a)/2)^{n+1}$ for $n\ = 0, 1, 2, 3, 4$.}
\label{fig:1.1}
\end{figure}

Our main result (Theorem \ref{thrm:1}) implies the following connection between the eigenvalues of the EDMD matrices and the
eigenvalues of the Koopman operator associated to an analytic full branch interval map. The EDMD matrices $\hat H^{(M)}_N$ and $\hat G^{(M)}_N$ here are defined as in Equation (\ref{eq:1.2}), with observables
given by monomials, $\psi_k(x) = x^k$ for $k = 0, \ldots, N-1$, and
a collection of equidistant nodes
$x_m=-1+\delta+2 m/M \in [-1, 1]$,
$m = 0, \ldots, M-1$, for any choice of $0\leq\delta\leq 2/M$.
Precise definitions and the proof of the result as Corollary \ref{coro:2} are given in Section \ref{sec:2}; for the moment we simply note that for $\rho>0$ we write 
$D_\rho=\{z\in \mathbb{C}:|z|<\rho\}$ for the open disk with radius $\rho$ centred at $0$. 

\begin{theorem*}
Let $T$ be an analytic full branch map on the interval $[-1,1]$,
whose inverse branches $\varphi_\ell, \ell=1,\ldots,d,$ extend analytically to
an open disk $D_R \subset \mathbb{C}$ with $\bigcup_\ell \varphi_\ell (D_R) \subseteq D_r$ for some $1 < r < R$, and assume $r/R < 1/\gamma$ with $\gamma = (1 + \sqrt 2)^2$. Then there exists a Hilbert space of holomorphic functions $\mathcal{H}$
with Banach space dual $\mathcal{H}'$, such that the Koopman operator $\mathcal{K} \colon \mathcal{H}' \to \mathcal{H}'$ is compact.

Moreover, for $M\geq N^2 R^N$ there exist enumerations of the eigenvalues of $\mathcal{K}$ and $(\hat H_N^{(M)})^{-1} \hat G_N^{(M)}$, $\lambda_k(\mathcal{K})$ and
$\lambda_k((\hat H_N^{(M)})^{-1} \hat G_N^{(M)})$ respectively, such that for each $k \in \mathbb{N}$, there are constants $C > 0$ and $b \in (0, 1)$ so that
\[
|\lambda_k((\hat H_N^{(M)})^{-1} \hat G_N^{(M)})-\lambda_k(\mathcal{K})|\leq C b^N \, .
\]
\end{theorem*}

In short, for analytic full branch interval maps this theorem guarantees exponential
convergence of eigendata of the EDMD matrix constructed using $N$
monomials as observables and $M$ equidistant points in the phase space, to those of
the associated Koopman operator (defined on a suitable function space), if $M$ satisfies a lower bound in terms of $N$.

We aim to keep our account self-contained and therefore start with an exposition that can
be found elsewhere in the literature and restate some well-known facts.
In particular, we introduce in some detail analytic
full branch maps and suitable function spaces for the associated transfer operators in Section~\ref{sec:2.1}, and reiterate some basic features of collocation errors
in Section~\ref{sec:2.2}. The main technical approximation lemmas are presented in Section~\ref{sec:2.3} culminating in the proof of our main Theorem~\ref{thrm:1} and its corollaries in Section~\ref{sec:2.4}.
The given estimates are conservative, and a refined approach may allow to
relax some of the constraints, such as the constraint on the expansion rate of the map, or
the relation between the number of nodes and the number of observables. Therefore we also provide a numerical analysis of EDMD in the context of linear and nonlinear analytic full branch maps of the interval in Section~\ref{sec:3}, focusing on maps where
the spectrum of the transfer operator is known explicitly, and where we are able to compare EDMD results with the exact expected outcome. These computations indicate that one may indeed be able to derive improved estimates with different tools. Finally, in Section~\ref{sec:3.3} we return
to the conundrum described above for the choice of Fourier modes presented in the left panel of Figure~\ref{fig:1.1}.

\section{Error estimates for EDMD}\label{sec:2}

In this section we provide a detailed proof of the convergence
properties of EDMD for certain one-dimensional chaotic maps.
To keep our presentation self-contained we will include some well-known
details and classical facts.
Initially, we set up the spectral theory for the dynamical system
we consider, and then move on to provide a detailed account
of the proof of Theorem \ref{thrm:1} and of the corollaries.

\subsection{Analytic full branch maps and their transfer operators}\label{sec:2.1}

The following are the standing assumptions we are going to use
throughout our exposition. We consider an analytic
full branch map $T\colon I \to I$ on the interval $I=[-1,1]$. By that we mean
there exists a collection of closed intervals $\{I_{\ell}:1\leq \ell \leq d \}$
with disjoint interiors, such that $I = \bigcup_\ell I_\ell$,
and $T\vert_{\operatorname{int}(I_\ell)}$ is an analytic diffeomorphism with $\overline{T(I_{\ell})}=I$ for each $\ell$.
We denote by $\varphi_\ell \colon I \rightarrow I_{\ell}$ the analytic inverse
branches of the map $T$, and assume that for some $R > 1$, each $\varphi_\ell$ has an analytic
extension to an open disk $D_R \subset \mathbb{C}$ with radius $R$ centred at $0$. More precisely, we require that $\varphi_\ell \in H^\infty(D_R)$, where
\begin{equation*}
H^\infty(D_R) =\{ f\in \mbox{Hol}(D_R) : \sup_{z\in D_R}|f(z)|<\infty \}
\end{equation*}
denotes the usual Banach space of holomorphic functions with the norm
\begin{equation*}
\|f\|_{H^\infty(D_R)} = \sup_{z\in D_R}|f(z)| \, .
\end{equation*}
Finally, we assume there exists $1<r<R$ such that
\begin{equation*}
\bigcup_{\ell=1}^d \varphi_\ell(D_R) \subseteq D_r \, ,
\end{equation*}
which poses a condition on the expansion of
the map $T$.

We associate with the map $T$ the transfer operator $\mathcal{L} \colon L^1(I) \to L^1(I)$ given by
\begin{equation}\label{eq:2.1.3}
(\mathcal{L} f)(z)=
\sum_{\ell=1}^d \sigma_\ell \varphi_\ell'(z) f(\varphi_\ell(z)),
\end{equation}
where $\sigma_\ell=\mbox{sgn}(\varphi_\ell'(0))=\pm 1$ encodes
whether the respective branch of the map is increasing or decreasing.
This operator, in this setting also known as the Perron-Frobenius operator, connects to the action of the map via the duality relation
\begin{equation}\label{eq:2.1.3a}
\int_{-1}^1 g(T(x)) f(x) \, dx = \int_{-1}^1 g(x) (\mathcal{L} f)(x) \, dx \,  
\quad (f \in L^1(I), g\in L^\infty(I)).
\end{equation}

We will first show that the expression (\ref{eq:2.1.3}) also gives rise to a well-defined and in fact compact
operator when viewed on certain spaces of analytic functions. The main result of this section will be Proposition \ref{lem:4}, showing that this operator is well-approximated by finite-rank Taylor approximations.

For that purpose we introduce the standard Hardy-Hilbert space
\begin{equation*}
H^2(D_R) =\{ f\in \mbox{Hol}(D_R) :
\sup_{s<R} \int_0^{2\pi} |f(s e^{i t})|^2 \, dt  <\infty \}
\end{equation*}
with norm
\begin{equation*}
\|f\|_{H^2(D_R)} = \sup_{s<R} \sqrt{\frac{1}{2\pi} \int_0^{2\pi} |f(s e^{i t})|^2 \,dt} \, .
\end{equation*}

Given $f\in \mbox{Hol}(D_R)$ we can write  $f(z)=\sum_{n=0}^\infty f_n z^n$, where $f_n$ denotes the $n$-th Taylor coefficient of $f$. For fixed
$N\in \mathbb{N}_0$ we let
\begin{equation}\label{eq:2.1.6}
(\mathcal{P}_N f)(z)=\sum_{n=0}^{N-1} f_n z^n
\end{equation}
denote the usual Taylor projection. Later on we will be more specific about the domain and codomain of $\mathcal{P}_N$.
For later use we note that $f\in H^2(D_R)$ if and only if $\sum_{n=0}^\infty |f_n|^2R^{-2n}<\infty$, in which case
\[ \|f\|^2_{H^2(D_R)}=\sum_{n=0}^\infty |f_n|^2R^{-2n}.  \]

We start by recalling the relation between the two classes of spaces introduced above in the following two lemmas.

\begin{lemma}\label{lem:1}
Let $\rho \in (r,\infty)$. Then the embedding
$\mathcal{J}_1 \colon H^2(D_\rho) \rightarrow H^\infty(D_r)$ is compact
and $\|\mathcal{J}_1\|_{H^2(D_\rho) \rightarrow H^\infty(D_r)}\leq \rho/\sqrt{\rho^2-r^2}$.
Moreover,
\[
\| \mathcal{J}_1 - \mathcal{P}_N\|_{H^2(D_\rho)\rightarrow H^\infty(D_r)}
\leq \frac{\rho}{\sqrt{\rho^2-r^2}} \left(\frac{r}{\rho}\right)^N \qquad \text{ for } N\in \mathbb{N}_0 \, ,
\]
where $\mathcal{P}_N \colon H^2(D_\rho) \rightarrow H^\infty(D_r)$ denotes the Taylor
projection defined in (\ref{eq:2.1.6}).
\end{lemma}

\begin{proof}
Let $f\in H^2(D_\rho)$. Then we can write $f(z)=\sum_{n=0}^\infty f_n z^n$ for
$z\in D_\rho$. Moreover, for $z\in D_r$ we have
\begin{multline*}
\left|\left((\mathcal{J}_1-\mathcal{P}_N)f\right)(z)\right|^2 =\left|\sum_{n=N}^\infty f_n z^n\right|^2
=\left|\sum_{n=N}^\infty f_n \rho^n \left(\frac{z}{\rho}\right)^n\right|^2
\\
\leq \sum_{n=N}^\infty |f_n|^2 \rho^{2n} \sum_{n=N}^\infty \left|\frac{z}{\rho}\right|^{2n}
\leq \|f\|^2_{H^2(D_\rho)} \left(\frac{r}{\rho}\right)^{2N} \frac{1}{1-(r/\rho)^2} \, .
\end{multline*}
Thus
\[
\|(\mathcal{J}_1-\mathcal{P}_N)f\|_{H^\infty(D_r)} \leq \|f\|_{H^2(D_\rho)} \frac{\rho}{\sqrt{\rho^2-r^2}}
\left(\frac{r}{\rho}\right)^N  \, .
\]
The estimate for the operator norm follows by taking $N=0$ in the above. Finally, $\mathcal{J}_1$ is compact, since it is a uniform limit
of finite-rank operators.
\end{proof}

\begin{lemma}\label{lem:2}
Let $\rho \in (0,R)$. Then the embedding
$\mathcal{J}_2 \colon H^\infty(D_R) \rightarrow H^2(D_\rho)$ is compact and
$\|\mathcal{J}_2\|_{H^\infty(D_R) \rightarrow H^2(D_\rho)}\leq 1$.
Moreover,
\[
\| \mathcal{J}_2 - \mathcal{P}_N\|_{H^\infty(D_R)\rightarrow H^2(D_\rho)}
\leq \left(\frac{\rho}{R}\right)^N \qquad \text{ for } N\in \mathbb{N}_0 \, ,
\]
where $\mathcal{P}_N \colon H^\infty(D_R) \rightarrow H^2(D_\rho)$ denotes the Taylor
projection  defined in (\ref{eq:2.1.6}).
\end{lemma}

\begin{proof}
Let $f\in H^\infty(D_R)$. Then we can write $f(z)=\sum_{n=0}^\infty f_n z^n$ for
$z\in D_R$. For $N \in \mathbb{N}_0$, we have
\begin{multline*}
\|(\mathcal{J}_2 -\mathcal{P}_N)f\|^2_{H^2(D_\rho)} =
\sup_{s<\rho} \frac{1}{2\pi} \int_0^{2\pi} \left|\sum_{n=N}^\infty f_n s^n e^{int} \right|^2 \, dt
= \sup_{s<\rho} \sum_{n=N}^\infty |f_n|^2 s^{2n}
= \sum_{n=N}^\infty |f_n|^2 R^{2n} \left(\frac{\rho}{R}\right)^{2n} \\
\leq \left(\frac{\rho}{R}\right)^{2N} \sum_{n=R}^\infty |f_n|^2 R^{2n}
\leq
\left(\frac{\rho}{R}\right)^{2N} \|f\|^2_{H^2(D_R)} \leq
\left(\frac{\rho}{R}\right)^{2N} \|f\|^2_{H^\infty(D_R)} \, .
\end{multline*}
As in the proof of Lemma~\ref{lem:1},
the estimate for the operator norm follows by taking $N=0$ in the above,
and $\mathcal{J}_2$ is compact since it is a uniform limit
of finite-rank operators.
\end{proof}

We now turn to the transfer operator defined in (\ref{eq:2.1.3}). Our ultimate aim is to show that it is compact and 
is well-approximated by operators of finite rank. We start with the following simple observation.

\begin{lemma}\label{lem:3}
The transfer operator (\ref{eq:2.1.3}) extends to a bounded
operator $\tilde{\mathcal{L}} \colon H^\infty(D_r) \rightarrow H^\infty(D_R)$
with
\[
\| \tilde{\mathcal{L}} \|_{H^\infty(D_r) \rightarrow H^\infty(D_R)} \leq \sup_{z\in D_R}
\sum_{\ell=1}^d |\varphi'_\ell(z)|  \, .
\]
\end{lemma}

\begin{proof}
Let $f\in H^\infty(D_r)$ and $z \in D_R$. Then
\[
|(\tilde{\mathcal{L}} f)(z)| \leq \sum_{\ell=1}^d |\varphi'_\ell(z)| |f(\varphi_\ell(z))|
\leq \sum_{\ell=1}^d |\varphi'_\ell(z)| \|f\|_{H^\infty(D_r)}
\]
so
\[
\| \tilde{\mathcal{L}} f \|_{H^\infty(D_R)} \leq \sup_{z \in D_R}
\sum_{\ell=1}^d |\varphi'_\ell(z)| \|f\|_{H^\infty(D_r)} \,
\]
and the assertion follows.
\end{proof}
For variants of this result for transfer operators arising from maps with infinitely many branches, see \cite{BJLMS_07}.

Now let $\mathcal{J}_1 \colon H^2(D_\rho) \rightarrow H^\infty(D_r)$
and
$\mathcal{J}_2 \colon H^\infty(D_R) \rightarrow H^2(D_\rho)$
denote the canonical
embeddings and $\tilde{\mathcal{L}}$ the transfer operator as an operator
$\tilde{\mathcal{L}} \colon H^\infty(D_r) \rightarrow H^\infty(D_R)$. Then the
transfer operator $\mathcal{L}$ viewed as an operator
$\mathcal{L} \colon H^2(D_\rho) \rightarrow H^2(D_\rho)$ factorises as
$\mathcal{L}=\mathcal{J}_2 \tilde{\mathcal{L}} \mathcal{J}_1$.
Using this factorisation we are now able to estimate the accuracy of
particular finite-rank approximations of the transfer operator.

\begin{proposition}\label{lem:4}
Let $\rho\in(r,R)$. Then, for $N\in \mathbb{N}_0$ we have 
\begin{align*}
\|\mathcal{L}-\mathcal{P}_N \mathcal{L}\|_{H^2(D_\rho)\rightarrow H^2(D_\rho)}
& \leq C
\left(\frac{\rho}{R}\right)^N, \\
\|\mathcal{L}-\mathcal{L} \mathcal{P}_N\|_{H^2(D_\rho)\rightarrow H^2(D_\rho)}
& \leq C
\left(\frac{r}{\rho}\right)^N,
\end{align*}
where
\[
C=\frac{\rho}{\sqrt{\rho^2-r^2}} \sup_{z\in D_R} \sum_{\ell=1}^d | \varphi'_\ell(z)|
\]
and $\mathcal{P}_N \colon H^2(D_\rho)\rightarrow H^2(D_\rho)$ denotes the Taylor projection.
In particular
\[
\| \mathcal{L}- \mathcal{P}_N \mathcal{L} \mathcal{P}_N \|_{H^2(D_\rho)\rightarrow H^2(D_\rho)}
\leq C \left( \left(\frac{\rho}{R}\right)^N + \left(\frac{r}{\rho}\right)^N \right) \, 
\]
for $N\in\mathbb{N}_0$. Choosing $\rho=\sqrt{R r}$ optimises the above convergence rate to
\[
\| \mathcal{L}- \mathcal{P}_N \mathcal{L} \mathcal{P}_N \|_{H^2(D_\rho)\rightarrow H^2(D_\rho)}
\leq 2 C \left(\frac{r}{R}\right)^{N/2} \, .
\]
\end{proposition}

\begin{proof}
Using $\mathcal{L} = \mathcal{J}_2 \tilde{\mathcal{L}} \mathcal{J}_1$,
by Lemmas \ref{lem:1}, \ref{lem:2}, and \ref{lem:3} we have
\begin{align*}
& \hphantom{\, = \ }
\|\mathcal{L}-\mathcal{P}_N \mathcal{L}\|_{H^2(D_\rho)\rightarrow H^2(D_\rho)}\\
& = \|(\mathcal{J}_2 -\mathcal{P}_N \mathcal{J}_2)\tilde{\mathcal{L}} \mathcal{J}_1\|_{H^2(D_\rho)\rightarrow H^2(D_\rho)}\\
& \leq \|\mathcal{J}_2-\mathcal{P}_N\|_{H^\infty(D_R)\rightarrow H^2(D_\rho)} \cdot
\|\tilde{\mathcal{L}}\|_{H^\infty(D_r)\rightarrow H^\infty(D_R)} \cdot
\|\mathcal{J}_1\|_{H^2(D_\rho)\rightarrow H^\infty(D_r)}\\
& \leq \left(\frac{\rho}{R}\right)^N \sup_{z\in D_R}\sum_{\ell=1}^d | \varphi'_\ell(z)| \frac{\rho}{\sqrt{\rho^2-r^2}} \, ,
\end{align*}
and, similarly,
\begin{align*}
& \hphantom{\, = \ }
\|\mathcal{L} -\mathcal{L} \mathcal{P}_N\|_{H^2(D_\rho)\rightarrow H^2(D_\rho)}\\ & =
\|\mathcal{J}_2 \tilde{\mathcal{L}} (\mathcal{J}_1-\mathcal{J}_1 \mathcal{P}_N)\|_{H^2(D_\rho)\rightarrow H^2(D_\rho)}\\
& \leq  \|\mathcal{J}_2\|_{H^\infty(D_R)\rightarrow H^2(D_\rho)} \cdot
\|\tilde{\mathcal{L}}\|_{H^\infty(D_r)\rightarrow H^\infty(D_R)}\cdot
\|\mathcal{J}_1-\mathcal{P}_N\|_{H^2(D_\rho)\rightarrow H^\infty(D_r)}\\
& \leq  \left(\frac{r}{\rho}\right)^N \sup_{z\in D_R}\sum_{\ell=1}^d | \varphi'_\ell(z)| \frac{\rho}{\sqrt{\rho^2-r^2}} \, .
\end{align*}

The rest follows by observing that
\[
\mathcal{L}-\mathcal{P}_N \mathcal{L} \mathcal{P}_N
=(\mathcal{L}-\mathcal{P}_N \mathcal{L}) + \mathcal{P}_N(\mathcal{L}-\mathcal{L} \mathcal{P}_N)
\]
together with the fact that $\|\mathcal{P}_N\|_{H^2(D_\rho)\rightarrow H^2(D_\rho)}=1$, which follows
since $\mathcal{P}_N$ is a self-adjoint projection on $H^2(D_\rho)$.
\end{proof}

\subsection{Collocation errors for EDMD}\label{sec:2.2}

In this section we leverage several classical facts to prove
Proposition \ref{lem:7},
providing basic bounds on the difference between the EDMD matrices in Equation (\ref{eq:1.2}) with a finite set of equidistant nodes, 
and their formal infinite-data limits.
The results are stated for
expanding analytic full branch maps, but in fact the analyticity assumption can be weakened to piecewise $C^1$.

Let $M \in \mathbb{N}$. For $m\in\{0,1,\ldots, M-1\}$ and
$0\leq\delta\leq 2/M$ let  $x_m=-1+\delta+2 m/M$ denote a
collection of equidistant nodes in $I=[-1,1]$. Given $N\in \mathbb{N}$ and
 $k,\ell\in \{0,\ldots,N-1\}$ let
\begin{equation}\label{eq:2.2.1}
\left(H_N^{(M)}\right)_{k\ell}= \frac{1}{M}\sum_{m=0}^{M-1} x_m^k x_m^\ell,\quad
\left(G_N^{(M)}\right)_{k\ell}=\frac{1}{M} \sum_{m=0}^{M-1} \left(T(x_m)\right)^k x_m^\ell,
\end{equation}
and define the $N\times N$ square matrices $\hat{H}_N^{(M)}$ and
$\hat{G}_N^{(M)}$ by $(\hat{H}_N^{(M)})_{k\ell}=(H_N^{(M)})_{k\ell}$ and
$(\hat{G}_N^{(M)})_{k\ell}=(G_N^{(M)})_{k\ell}$. These matrices
determine EDMD for $M$ equidistant nodes and $N$ observables, where the observables are given by
monomials (see Equation (\ref{eq:1.2})).

In order to study the impact of the number of nodes we also introduce for $k,\ell \in \mathbb{N}_0$
\begin{equation}\label{eq:2.2.2}
H_{k\ell} =\frac{1}{2} \int_{-1}^1 x^k x^\ell \, dx, \quad
G_{k\ell}=\frac{1}{2} \int_{-1}^1 \left(T(x)\right)^m x^\ell \, dx
\end{equation}
and define the two $N\times N$ square matrices $\hat{H}_N$ and $\hat{G}_N$ by
$\left(\hat{H}_N\right)_{k\ell}=H_{k\ell}$ and
$\left(\hat{G}_N\right)_{k\ell}=G_{k\ell}$ for $k,\ell\in \{0,\ldots, N-1\}$.
These matrices may be viewed as a formal $M\rightarrow \infty$ limit of the
finite sums occurring in EDMD and we shall sometimes refer to these objects as EDMD matrices with `infinitely many nodes'.
The following lemmas provide the
collocation errors of the expressions just introduced.

We begin by recalling the following standard bound for the error in the Riemann integral.

\begin{lemma}\label{lem:5}
Let $f \colon [a,b]\rightarrow \mathbb{R}$ be piecewise $C^1$, that is, $f$ is continuously
differentiable apart from a finite set of points
$S=\{s_1,s_2,\ldots, s_d\}\subset [a,b]$ with $s_1<s_2<\ldots<s_d$, and $f$ has an extension
to a $C^1$ function on $[s_{\ell-1},s_\ell]$ for $\ell \in \{1,\ldots, d+1\}$ where we set $s_0=a$ and $s_{d+1}=b$. Given $M\in\mathbb{N}$, define $h=(b-a)/M$ and
$x_k=a+k h$ for $k\in \{0,\ldots, M\}$. For any $t_k\in [x_{k-1},x_k]$ and $k \in \{1,\ldots, M\}$ we have
\[
\left| \int_a^b f(x) \, dx -\sum_{k=1}^M f(t_k) h\right| \leq \frac{1}{2} \| f'\| \frac{(b-a)^2}{M}
+2 d \| f\| \frac{b-a}{M}
\]
where
\[
\| f'\|  =\sup_{x\in[a,b]\backslash S} |f'(x)|, \quad
\| f \|=\sup_{x\in [a,b]}|f(x)| \, .
\]
\end{lemma}

\begin{proof}
We start by showing that
\begin{equation}\label{eq:2.2.3}
\left|\int_{x_{k-1}}^{x_k} f(x) \, dx - f(t_k) h\right| \leq \frac{1}{2} \| f'\| h^2 + 2\chi_k \| f\|  h
\end{equation}
where
\[
\chi_k =\left\{\begin{array}{rcl}
1 & \mbox{ if } & (x_{k-1},x_k) \cap S \neq \emptyset\\
0 & \mbox{ if } & (x_{k-1},x_k) \cap S = \emptyset
\end{array} \right. \, .
\]
In order to see this, note that $(x_{k-1},x_k) \cap S$ is either empty or finite. If it is empty, then
\begin{multline*}
\left|\int_{x_{k-1}}^{x_k} f(x) \, dx - f(t_k) h\right| =
\left| \int_{x_{k-1}}^{x_k} (f(x)-f(t_k)) \, dx \right| \\
\leq \int_{x_{k-1}}^{x_k} | f(x)-f(t_k)| \, dx \leq \| f'\| \int_{x_{k-1}}^{x_k}|x-t_k| \, dx
\leq \frac{1}{2} h^2 \| f'\|
\end{multline*}
by the Mean Value Theorem, and (\ref{eq:2.2.3}) follows in this case.
If $(x_{k-1},x_k) \cap S \neq \emptyset$, then we use
\[
\left|\int_{x_{k-1}}^{x_k} f(x) \, dx - f(t_k) h\right| \leq 2 \|f\| h\,,
\]
which proves (\ref{eq:2.2.3}) in this case.
Now, using (\ref{eq:2.2.3}) we have
\begin{multline*}
 \left| \int_a^b f(x) \, dx -\sum_{k=1}^M f(t_k) h\right|
= \left| \sum_{k=1}^M \left( \int_{x_{k-1}}^{x_k} f(x) \, dx - f(t_k) h\right) \right|\\
\leq \frac{M}{2} \|f'\| h^2 + \sum_{k=1}^M \chi_k \cdot 2 \|f\| h \leq
\frac{1}{2}\|f'\| \frac{(b-a)^2}{M} + 2 d \|f \| \frac{b-a}{M} \, .
 \qedhere
\end{multline*}
\end{proof}

Our subsequent estimates for the matrix norms will rely on a famous
lemma of Schur, the short proof of which we recall for the convenience of the reader.

\begin{lemma}\label{lem:6}
Let $A\in \mathbb{C}^{N\times N}$ be a complex $N\times N$ matrix and write
\[
R=\sum_{m=1}^N \max_{1\leq n\leq N} |A_{mn}|, \quad
C=\sum_{n=1}^N \max_{1\leq m\leq N} |A_{mn}| \, .
\]
Then the spectral norm $\|A\|_2$ of $A$, that is, the matrix norm induced by the Euclidean norm, satisfies
$\| A\|_2 \leq \sqrt{CR}$.
\end{lemma}

\begin{proof}
Let $x\in\mathbb{C}^N$ and $y\in \mathbb{C}^N$. Then by the Cauchy-Schwarz inequality
\begin{align*}
\left|\sum_{k,\ell=1}^N \bar{y}_k  A_{k\ell} x_\ell \right|^2
&\leq \left(\sum_{k,\ell=1}^N |y_k||A_{k\ell}|^{1/2} \cdot |A_{k\ell}|^{1/2} |x_\ell|\right)^2 \\
&\leq \sum_{k,\ell=1}^N  |y_k|^2 |A_{k\ell}| \cdot \sum_{k,\ell=1}^N |A_{k\ell}| |x_\ell|^2
\leq (\|y\|_2)^2 C \cdot R (\| x\|_2)^2 \, .
\end{align*}
Setting $y=Ax$ we obtain $(\|A x\|_2)^4\leq C R (\|x\|_2)^2 (\| A x\|_2)^2$, from which the assertion follows.
\end{proof}

We are now able to bound the collocation error caused by taking a finite
number of nodes.

\begin{proposition}\label{lem:7}
Denote by $T$ an analytic full branch map on $I=[-1,1]$ with $d$ branches,
and let $S$ be the set of the $d-1$ `critical points' of the map $T$.
For $k,\ell \in \{0,\ldots,N-1\}$ we have
\begin{align*}
\left|\left(\hat{H}_N-\hat{H}_N^{(M)}  \right)_{k\ell}\right| &\leq \frac{k+\ell}{M} \, , \\
\left|\left(\hat{G}_N-\hat{G}_N^{(M)}  \right)_{k\ell}\right| &\leq \max(\|T'\|,2(d-1)) \frac{k+l+1}{M}  \, ,
\end{align*}
where $\|T'\|=\sup_{x\in[-1,1]\backslash S}|T'(x)|$.

Moreover
\begin{align*}
\left\| \hat{H}_N-\hat{H}_N^{(M)}\right\|_2 &\leq \frac{3}{2} \frac{N^2}{M} \, ,\\
\left\| \hat{G}_N-\hat{G}_N^{(M)}\right\|_2 &\leq \frac{3}{2} \max(\|T'\|,2(d-1)) \frac{N^2}{M} \, .
\end{align*}
\end{proposition}

\begin{proof}
Using Lemma \ref{lem:5} with $f(x)=x^{k+\ell}$ we have
\[
\left|\left(\hat{H}_N-\hat{H}_N^{(M)} \right)_{k \ell}\right| =
\left|\frac{1}{2} \int_{-1}^1 f(x) \, dx-\frac{1}{M} \sum_{m=0}^{M-1}f(x_m) \right|
\leq \frac{1}{2}\left(\frac{1}{2}\|f'\|\cdot\frac{2^2}{M}\right)=\frac{k+\ell}{M},
\]
since $f'(x)=(k+\ell)x^{k+\ell-1}$ and hence $\|f'\|=k+\ell$.

Next we set $f(x)=(T(x))^k x^\ell$ and observe that $f'(x)=k (T(x))^{k-1} T'(x) x^\ell+
(T(x))^k \cdot \ell x^{\ell-1}$, and therefore $\|f'\|\leq k
\|T'\|+\ell\leq (k+\ell)\|T'\|$, using that $\|T'\| \geq 1$.
Lemma \ref{lem:5} now gives
\begin{multline*}
\left|\left(\hat{G}_N-\hat{G}_N^{(M)}\right)_{k\ell} \right| = \left| \frac{1}{2}
\int_{-1}^1 f(x) \, dx - \frac{1}{M}\sum_{m=0}^{M-1} f(x_m) \right|
\leq \frac{1}{2}\left(\frac{1}{2} \|f'\|\frac{2^2}{M}+ 2(d-1) \|f\| \frac{2}{M}\right)\\
=\frac{1}{2}\left( 2 \|T'\| \frac{k+\ell}{M}+\frac{2^2 (d-1)}{M}\right)
\leq \max(\|T'\|,2(d-1)) \frac{k+\ell+1}{M} \,.
\end{multline*}

The remaining assertions follow from Lemma \ref{lem:6}. For instance, if $B\in \mathbb{C}^{N\times N}$ with
$|B_{k\ell}|\leq k+\ell+1$ for $k,\ell \in \{0,\ldots,N-1\}$, then $\max_k |B_{k\ell}|\leq N+\ell$ and
\[
\sum_{\ell=0}^{N-1} \max_{0\leq k\leq N-1} |B_{k\ell}|\leq \sum_{\ell=0}^{N-1}(N+\ell)=N^2+
\frac{1}{2}N(N-1)\leq \frac{3}{2} N^2 \, .
\]
By symmetry $\sum_{k=0}^{N-1}\max_\ell |B_{k\ell}|\leq 3N^2/2$. Thus by Lemma \ref{lem:6}, we have
$\|B\|_2\leq 3N^2/2$.
\end{proof}

\subsection{EDMD in an operator setting}\label{sec:2.3}

This section provides the key technical elements of our paper, preparing the proof of our main results. Proposition \ref{lem:11} provides a bound on the Galerkin approximation error, that is, the error incurred from the finite-rank approximation of the transfer operator. Complementarily, in Proposition \ref{lem:13} we use the basic bounds from
the previous section to establish a bound on the collocation error incurred by the finite-node approximation of the integrals in Equation (\ref{eq:2.2.2}).

The following calculations will be involving matrix representations
of operators on $H^2(D_\rho)$ with some fixed $\rho\in(r,R)$ with respect
to the standard orthonormal basis $(e_n)_{n\in\mathbb{N}_0}$ given by
\begin{equation}\label{eq:2.2.4}
e_n(z)=\left(\frac{z}{\rho}\right)^n\, \qquad \text{ for } n\in \mathbb{N}_0 \, .
\end{equation}
We will denote by
\begin{equation*}
(f,g)= \lim_{s\uparrow\rho} \frac{1}{2\pi} \int_0^{2\pi} f(s \exp(it)) \overline{g(s \exp(it))} \,dt
\end{equation*}
the standard inner product in $H^2(D_\rho)$, and by
\begin{equation}\label{eq:2.2.6}
\langle f, g \rangle = \frac{1}{2} \int_{-1}^1 f(x) \overline{g(x)} \, dx
\end{equation}
the standard inner product in $L^2([-1,1])$.

With the orthonormal basis above we can write the transfer operator
$\mathcal{L} f =\sum_{k\ell} (f,e_\ell) L_{k\ell} e_k$ with
\begin{equation}\label{eq:2.2.7}
L_{k \ell}=(\mathcal{L} e_\ell,e_k) \, .
\end{equation}
Equation (\ref{eq:2.2.7}) defines a linear operator $L \colon \ell^2(\mathbb{N}_0) \rightarrow \ell^2(\mathbb{N}_0)$,
where $\|\mathcal{L}\|_{H^2(D_\rho)}=\|L\|_{\ell^2}$. We use Equations (\ref{eq:2.2.2}) to consider
$G$ and $H$ as densely defined operators on $\ell^2(\mathbb{N}_0)$. We will also need
the following densely defined diagonal operator on $\ell^2(\mathbb{N}_0)$
\begin{equation}\label{eq:2.2.8}
V \colon (x_n)_{n\in \mathbb{N}_0} \mapsto (\rho^n x_n)_{n\in \mathbb{N}_0}
\end{equation}
and its inverse $V^{-1} \colon \ell^2(\mathbb{N}_0) \rightarrow \ell^2(\mathbb{N}_0)$
\begin{equation}\label{eq:2.2.9}
V^{-1} \colon (x_n)_{n\in \mathbb{N}_0} \mapsto (\rho^{-n} x_n)_{n\in \mathbb{N}_0},
\end{equation}
which is a compact operator since $\rho>1$. While $V$ and $V^{-1}$ depend on $\rho$, we omit to explicitly write out this dependence for notational simplicitly.

We are now going to link the expressions defining EDMD with the transfer
operator, and the following lemma is the first in a couple of steps.

\begin{lemma}\label{lem:8}
With the above notation, we have $G=H V^{-1} L V$.
\end{lemma}

\begin{proof}
By Equation (\ref{eq:2.2.4}) we have $x^k=\rho^k e_k(x)$, so that Equations (\ref{eq:2.2.2}) can be written as
\begin{equation*}
H_{k\ell}=\rho^{k+\ell} \langle e_\ell, e_k\rangle,
\quad G_{k\ell}=\rho^{k+\ell} \langle \mathcal{L} e_\ell, e_k\rangle
\end{equation*}
recalling Equations (\ref{eq:2.1.3a}) and (\ref{eq:2.2.6}).
Since $\mathcal{L} e_\ell = \sum_n (\mathcal{L} e_\ell, e_n) e_n$, we obtain
\begin{multline*}
G_{k\ell}=\rho^{k+\ell}\langle \mathcal{L} e_\ell, e_k\rangle
= \sum_{n\in \mathbb{N}_0} \rho^{k+\ell} (\mathcal{L} e_\ell,e_n) \langle e_n,e_k\rangle\\
=\sum_{n \in \mathbb{N}_0} \rho^{k+\ell} L_{n\ell} \rho^{-n-k} H_{k n}
= \sum_{n \in \mathbb{N}_0} H_{k n} \rho^{-n} L_{n\ell} \rho^\ell \, .
\qedhere
\end{multline*}
\end{proof}

We are now going to introduce operators on $\ell^2(\mathbb{N}_0)$ which capture
a finite number of observables. Let $P_N \colon \ell^2(\mathbb{N}_0)\rightarrow \ell^2(\mathbb{N}_0)$
denote the canonical orthogonal projection
onto the first $N$ coordinates which gives the matrix representation of the Taylor projection
$\mathcal{P}_N\colon H^2(D_\rho) \rightarrow H^2(D_\rho)$ with respect to the canonical orthogonal basis
$(e_n)_{n\in\mathbb{N}_0}$. Let $L_N$, $H_N$ and $G_N$ denote the orthogonal projections of the
operators introduced above, that is,
\begin{equation}\label{eq:2.2.10}
L_N=P_N L P_N, \quad H_N= P_N H P_N, \quad G_N =P_N G P_N
\end{equation}
defined as operators of rank at most $N$ on $\ell^2(\mathbb{N}_0)$.
If we restrict $H_N$ and $G_N$ on the image of $P_N$, that is, if we consider
the $N\times N$ section, we obtain the
$N\times N$ square matrices $\hat{H}_N$ and $\hat{G}_N$ mentioned previously.
It turns out that $\hat{H}_N$ is invertible (because it is positive definite)
and we denote by $H^{\dagger}_N \colon \ell^2(\mathbb{N}_0)
\rightarrow \ell^2(\mathbb{N}_0)$ the rank $N$ operator which
lifts the matrix $\hat{H}_N^{-1}$ to $\ell^2(\mathbb{N}_0)$. That means
$H^{\dagger}_N$ coincides with $\hat{H}_N^{-1}$ when
restricted to the image of $P_N$, that is, $H_N H^{\dagger}_N=H^{\dagger}_N H_N=P_N$ and
$H^{\dagger}_N P_N=P_N H^{\dagger}_N=H_N^{\dagger}$. We note that since $H_N$
is symmetric, $H_N^{\dagger}$ can be seen to be the Moore-Penrose pseudoinverse of $H_N$.

\begin{lemma}\label{lem:9} For all $N\in \mathbb{N}_0$ we have
\[
\|H_N\|_{\ell^2}\leq \|H\|_{\ell^2} \leq \pi \quad \text{ and } \quad
\|H_N^{\dagger}\|_{\ell^2}\leq C (1+\sqrt{2})^{2 N}
\]
for some $C>0$.
\end{lemma}

\begin{proof}
Let $F_{k\ell}=1/(k+\ell+1)$ with $k,\ell \in \mathbb{N}_0$ denote the Hilbert matrix.
It is known (see, for example, \cite[p.305]{Choi_AMM83}) that $\|F\|_{\ell^2}=\pi$. Now
\[
H=\frac{1}{2} \left(F + J F J\right)
\]
where $J \colon \ell^2(\mathbb{N}_0) \rightarrow \ell^2(\mathbb{N}_0) $ is given by
$(Jx)_k = (-1)^k x_k$. Since $\|J\|_{\ell^2}=1$ we have
\[
\| H\|_{\ell^2} \leq \|F\|_{\ell^2}=\pi \, .
\]
Furthermore, with $\| P_N\|_{\ell^2}=1$ we obtain
\[
\| H_N\|_{\ell^2} = \|P_N H P_N \|_{\ell^2} \leq \| H \|_{\ell^2} \, .
\]
The second inequality follows from \cite[Theorem~3.2]{Wilf:70}.
\end{proof}

We also introduce orthogonal projections of the diagonal operators defined
in Equations (\ref{eq:2.2.8}) and (\ref{eq:2.2.9})
\begin{equation}\label{eq:2.2.11}
V_N=P_N V P_N, \quad V_N^{-1}=P_N V^{-1} P_N \, .
\end{equation}

\begin{lemma}\label{lem:10}
For all $N\in \mathbb{N}_0$ we have
\[
\| V_N H_N^{\dagger} G_N V_N^{-1}-L_N\|_{\ell^2} \leq C \left(\frac{\gamma \rho}{R}\right)^N
\]
for some $C>0$, where $\gamma=(1+\sqrt{2})^2$.
\end{lemma}

\begin{proof}
By Lemma \ref{lem:8} we have
\[
P_N G P_N = P_N H V^{-1} L V P_N =
P_N H P_N V^{-1} L V P_N +  P_N H (1-P_N) V^{-1} L V P_N \, ,
\]
so using Equations (\ref{eq:2.2.10}) and (\ref{eq:2.2.11}) we have
\[
G_N=H_N V_N^{-1} L_N V_N + P_N H (1-P_N) V^{-1} L P_N V_N\, ,
\]
since $P_N V^{-1}= V^{-1}_N P_N$ and $V P_N = P_N V_N$. Thus
using $H_N^{\dagger} H_N= P_N$ and $V_N V_N^{-1}=P_N$ we obtain
\begin{align*}
V_N H_N^{\dagger} G_N V_N^{-1}&=L_N+ V_N H_N^{\dagger} P_N H (1-P_N) V^{-1} L P_N\\
&= L_N+ V_N H_N^{\dagger} P_N H (1-P_N) V^{-1} (1-P_N) L P_N
\end{align*}
and
\begin{align*}
& \hphantom{\, =\ }
\| V_N H_N^{\dagger} G_N V_N^{-1} -L_N\|_{\ell^2}\\
&\leq
\|V_N\|_{\ell^2}  \|H_N^{\dagger}\|_{\ell^2 } \| P_N H (1-P_N)\|_{\ell^2}
\| (1-P_N) V^{-1} (1-P_N)\|_{\ell^2} \|(1-P_N) L P_N\|_{\ell^2} \, .
\end{align*}
On the other hand, by definition of the diagonal operator we have
\[
\|V_N\|_{\ell^2}=\rho^{N-1}, \quad \|(1-P_N) V^{-1} (1-P_N)\|_{\ell^2}= \rho^{-N} \, ,
\]
by Lemma \ref{lem:9} we have
\[
\|P_N H (1-P_N)\|_{\ell^2} \leq 2 \pi, \quad \|H_N^{\dagger}\|_{\ell^2} \leq C_1 \gamma^N \, ,
\]
and Proposition \ref{lem:4} yields
\[
\|(1-P_N) L P_N\|_{\ell^2} \leq \|(1-P_N) L\|_{\ell^2} \| P_N \|_{\ell^2}
= \|(1-\mathcal{P}_N) \mathcal{L}\|_{H^2(D_\rho)\rightarrow H^2(D_\rho)}  \leq C_2 \left(\frac{\rho}{R}\right)^N
\]
with suitable $C_1,C_2>0$, and the assertion follows.
\end{proof}

\begin{proposition}\label{lem:11}
Writing $\gamma=(1+\sqrt{2})^2$, we have for all $N\in \mathbb{N}_0$
\[
\| V_N H_N^{\dagger} G_N V_N -L \|_{\ell^2}\leq C \left( \left(\frac{\gamma\rho}{R}\right)^N
+\left(\frac{r}{\rho}\right)^N \right)
\]
for some $C>0$.  Choosing
$\rho=\sqrt{rR/\gamma}$ optimises the convergence speed to
$(\gamma r/R)^{N/2}$.
\end{proposition}

\begin{proof}
The assertion follows immediately from Proposition \ref{lem:4} and Lemma \ref{lem:10}, as
\begin{align*}
\hphantom{=} {} & \|V_N H_N^{\dagger} G_N V_N^{-1} - L\|_{\ell^2}\\
\leq {} &
\|V_N H_N^{\dagger} G_N V_N^{-1} - L_N\|_{\ell^2} +\| L-L_N\|_{\ell^2} \\
= {} & \|V_N H_N^{\dagger} G_N V_N^{-1} - L_N\|_{\ell^2} +\| \mathcal{L}-\mathcal{P}_N
\mathcal{L}\mathcal{P}_N\|_{H^2(D_\rho)\rightarrow H^2(D_\rho)}\\
\leq {} &  C_1 \left(\frac{\gamma\rho}{R}\right)^N
+C_2 \left(\left(\frac{\rho}{R}\right)^N+\left(\frac{r}{\rho}\right)^N\right) \, . \qedhere
\end{align*}
\end{proof}

Proposition \ref{lem:11} is the first main convergence result. It shows that
for an infinite number of nodes
the EDMD matrix $H_N^{\dagger} G_N$ of order $N$, is similar to a matrix that converges
uniformly exponentially fast to the matrix representation of the transfer operator,
provided the map $T$ is sufficiently expansive, that is, if $r/R<1/\gamma$.

We will now implement the concept with a finite number of nodes in our setting.
For that purpose lift the $N\times N$ square matrix $\hat{H}_N^{(M)}$ (see Equation (\ref{eq:2.2.1}))
to a rank $N$ operator $H_N^{(M)} \colon \ell^2(\mathbb{N}_0) \rightarrow \ell^2(\mathbb{N}_0)$
such that the $N\times N$ section of $H_N^{(M)}$ coincides with $\hat{H}_N^{(M)}$ and
$H_N^{(M)}=P_N H_N^{(M)}=H_N^{(M)} P_N$. Furthermore we can see that
$\hat{H}_N^{(M)}$ is positive definite for $M\geq N$ since
\begin{equation}\label{eq:2.2.12}
\sum_{k,\ell=0}^{N-1} \bar{\alpha}_k (\hat{H}_N^{(M)})_{k \ell} \alpha_\ell
=\frac{1}{M}\sum_{m=0}^{M-1} |p(x_m)|^2\, ,
\end{equation}
where the polynomial $p(x)=\sum_{k=0}^N \alpha_k x^k$ has at most $N-1$ distinct roots
as long as $(\alpha_0,\ldots,\alpha_{N-1})$ is nonzero, so that at least one of the
terms in the sum of Equation (\ref{eq:2.2.12}) is positive. In the same vein,
provided that $M\geq N$, we can lift the inverse matrix $(\hat{H}_N^{(M)})^{-1}$ to a rank $N$ operator $(H_N^{(M)})^{\dagger} \colon
\ell^2(\mathbb{N}_0) \rightarrow \ell^2(\mathbb{N}_0)$ which obeys
$H_N^{(M)} (H_N^{(M)})^{\dagger}=(H_N^{(M)})^{\dagger} H_N^{(M)}=P_N$.

\begin{lemma}\label{lem:12}
There exist constants $C_1,C_2>0$ such that for $N\in \mathbb{N}_0$ and $M\geq C_1 N^2 \gamma^N$ we have
$$\|(H_N^{(M)})^{\dagger}\|_{\ell^2} \leq C_2 \gamma^N,$$ where $\gamma=(1+\sqrt{2})^2$.
\end{lemma}

\begin{proof}
We have $\| H_N^{\dagger}\|\leq \tilde{C}_1 \gamma^N$ by Lemma \ref{lem:9}, and
\[
\|H_N - H_N^{(M)}\|_{\ell^2}= \|\hat{H}_N-\hat{H}_N^{(M)}\|_2 \leq \frac{3}{2} \frac{N^2}{M}
\]
by Proposition \ref{lem:7}. Define $C_1= \max(3 \tilde{C}_1,1)$ and suppose that
$M\geq C_1 N^2 \gamma^N >N$. Then
\[
\|H_N^{\dagger} (H_N-H_N^{(M)})\|_{\ell^2}\leq \|H_N^{\dagger}\|_{\ell^2} \|H_N-H_N^{(M)}\|_{\ell^2}
\leq\frac{1}{2}
\]
and $I-H_N^{\dagger}(H_N-H_N^{(M)})$ is invertible. From
\[
H_N^{(M)}=H_N-(H_N-H_N^{(M)})=H_N-P_N (H_N-H_N^{(M)})=H_N(I-H_N^{\dagger}(H_N-H_N^{(M)}))
\]
we conclude (since $M>N$)
\[
(H_N^{(M)})^{\dagger}=(I-H_N^{\dagger}(H_N-H_N^{(M)}))^{-1} H_N^{\dagger}
\]
and
\[
\| (H_N^{(M)})^{\dagger}\|_{\ell^2} \leq \|(I-H_N^{\dagger}(H_N-H_N^{(M)}))^{-1}\|_{\ell^2}
\| H_N^{\dagger}\| \leq \frac{1}{1-1/2} \tilde{C}_1 \gamma^N \, . \qedhere
\]
\end{proof}

The following proposition essentially estimates the difference between EDMD with a finite
and with an infinite number of nodes.

\begin{proposition}\label{lem:13}
Let $r/R<1/\gamma$ with $\gamma=(1+\sqrt{2})^2$, and $\rho \in (r,R)$. There exist $C_1,C_2>0$ such that for 
$N\in \mathbb{N}_0$ and $M\geq C_1 N^2 \gamma^N$ we have
\[
\|V_N (H_N^{(M)})^{\dagger} G_N^{(M)} V_N^{-1}- V_N H_N^{\dagger} G_N V_N^{-1}\|_{\ell^2}
\leq C_2 (\rho \gamma)^N \frac{N^2}{M}\, .
\]
\end{proposition}

\begin{proof}
Let $C_1>0$ be as in Lemma \ref{lem:12} and $M\geq C_1 N^2 \gamma^N$. We have
\begin{align*}
\hphantom{=} {}& \|V_N (H_N^{(M)})^{\dagger} G_N^{(M)} V_N^{-1}- V_N H_N^{\dagger} G_N V_N^{-1}\|_{\ell^2}\\
= {} &\| V_N (H_N^{(M)})^{\dagger} (G_N^{(M)}-G_N) V_N^{-1}\\
 {} & \qquad +
V_N (H_N^{(M)})^{\dagger} (H_N^{(M)}-H_N) V_N^{-1} V_N H_N^{\dagger} G_N V_N^{-1}\|_{\ell^2}\\
\leq {} & \|V_N\|_{\ell^2} \|(H_N^{(M)})^{\dagger}\|_{\ell^2} \|G_N^{(M)}-G_N\|_{\ell^2} \|V_N^{-1}\|_{\ell^2}\\
{} & \qquad + \|V_N\|_{\ell^2} \|(H_N^{(M)})^{\dagger}\|_{\ell^2}
\|H_N^{(M)}-H_N\|_{\ell^2} \|V_N^{-1}\|_{\ell^2} \|V_N H_N^{\dagger} G_N V_N^{-1}\|_{\ell^2} \, .
\end{align*}

Using Proposition \ref{lem:7}, we obtain
\begin{align*}
\|H_N^{(M)}-H_N\|_{\ell^2} &= \|\hat{H}_N^{(M)}-\hat{H}_N\|_2\leq \frac{3}{2} \frac{N^2}{M}\\
\|G_N^{(M)}-G_N\|_{\ell^2} &= \|\hat{G}_N^{(M)}-\hat{G}_N\|_2\leq \tilde{C}_1 \frac{N^2}{M} \, ,
\end{align*}
and by Lemma \ref{lem:12},
\[
\| (H_N^{(M)})^{\dagger}\|_{\ell^2} \leq \tilde{C}_2 \gamma^N \, .
\]
Finally, Proposition \ref{lem:11} yields
\begin{align*}
\|V_N H_N^{\dagger} G_N V_N\|_{\ell^2} &= \|V_N H_N^{\dagger} G_N V_N-L\|_{\ell^2}+\|L\|_{\ell^2}\\
&\leq  \tilde{C}_3 \left(\frac{\gamma r}{R}\right)^{N/2} + \|\mathcal{L}\|_{H^2(D_\rho)\rightarrow H_2(D_\rho)}
\leq \tilde{C}_4 \, ,
\end{align*}
and
\[
\| V_N\|_{\ell^2} \|V_N^{-1}\|_{\ell2}=\rho^{N-1} \, ,
\]
and the assertion follows.
\end{proof}

\subsection{Main results}\label{sec:2.4}

Our first main result can be stated as follows:

\begin{theorem}\label{thrm:1}
Let $T$ be an analytic full branch map on the interval $[-1,1]$ with $r/R<1/\gamma$,
where $\gamma=(1+\sqrt{2})^2$. Then there exist
$C_1, C_2 >0$ such that for all $N\in \mathbb{N}_0$ and all $M\geq C_1 N^2 \gamma^N$ we have
\[
\| V_N (H_N^{(M)})^{\dagger} G_N^{(M)} V_N^{-1} -L\|_{\ell^2} \leq
C_2 \left( \frac{(\gamma r R)^{N/2} N^2}{M}+\left(\frac{\gamma r}{R}\right)^{N/2} \right) \, .
\]
\end{theorem}

\begin{proof}
We note that
\begin{align*}
\| V_N (H_N^{(M)})^{\dagger} G_N^{(M)} V_N^{-1} -L\|_{\ell^2} &\leq
\| V_N (H_N^{(M)})^{\dagger} G_N^{(M)} V_N^{-1} - V_N H_N^{\dagger} G_N V_N^{-1}\|_{\ell^2}\\
& \qquad +\|V_N H_N^{\dagger} G_N V_N^{-1}-L\| \, .
\end{align*}
Using Proposition \ref{lem:11} and Proposition \ref{lem:13} with $\rho=\sqrt{rR/\gamma}$
yields the assertion.
\end{proof}

In particular, by imposing a lower-bound condition on the number of nodes
in terms of the number of observables, we obtain exponential convergence. More precisely:

\begin{corollary}\label{coro:1}
Under the hypothesis of Theorem \ref{thrm:1}, there exists a $C > 0$ such that for all $N\in \mathbb{N}_0$ and for all $M\geq N^2 R^N$ we have
\[
\| V_N (H_N^{(M)})^{\dagger} G_N^{(M)} V_N^{-1} -L\|_{\ell^2} \leq C \left(\frac{\gamma r}{R}\right)^{N/2} \, .
\]
\end{corollary}

\begin{proof}
We note that by assumption $R^N > r^N \gamma^N$, and so $M \geq N^2 R^N \geq C_1 N^2 \gamma^N$
for all $N \geq \log C_1 / \log r$.
Letting
\[
C_3 = \max \left\{ \frac{\| V_N (H_N^{(M)})^{\dagger} G_N^{(M)} V_N^{-1} -L\|_{\ell^2}}{
\frac{(\gamma r R)^{N/2} N^2}{M}+\left(\frac{\gamma r}{R}\right)^{N/2}}  : N < \frac{\log C_1}{\log r}, M < C_1 N^2 \gamma^N \right \},
\]
the assertion follows with $C = \max \{2 C_2, C_3 \}$.
\end{proof}

It is a classical textbook result (see, for example, \cite[p.1091]{DuSw:88}) that
uniform convergence of compact operators, such as in
Corollary \ref{coro:1}, implies uniform convergence of eigenvalues.
However, without additional conditions, the speed of convergence does
not carry over. In our case, exponential convergence of eigenvalues is a
consequence of \cite[Theorem 2.18 and ensuing remarks]{AhLaLi:01}
(see also \cite{SlBaJu_CNSNS20}),
which provide the proof of the following simple useful statement.

\begin{lemma}\label{lem:14}
Let $V$ be a Banach space and $\mathcal{F}_N \colon V \rightarrow V$ a convergent
sequence of finite-rank operators with
\[
\|\mathcal{F}_N -\mathcal{F}\|_{V\rightarrow V} \leq C e^{-\alpha N}
\]
for some $C>0$ and $\alpha>0$. Then there exist enumerations
(taking algebraic multiplicities into account) of
the eigenvalues of $\mathcal{F}_N$ and of the limit $\mathcal{F}$, denoted
$\lambda_k(\mathcal{F}_N)$ and $\lambda_k(\mathcal{F})$ respectively, such that for
each fixed $k$ there exist constants $C'>0$ and $\alpha'>0$ with
\[
|\lambda_k(\mathcal{F}_N)-\lambda_k(\mathcal{F})|\leq C' e^{-\alpha' N} \qquad \text{ for }  N\in \mathbb{N} \, .
\]
\end{lemma}

With Corollary \ref{coro:1} and Lemma \ref{lem:14} the eigenvalues
of the EDMD matrices $(H_M^{(N)})^{\dagger} G_N^{(M)}$ converge at
an exponential rate towards the eigenvalues of the (compact) transfer
operator $\mathcal{L}$ on $H^2(D_\rho)$. If we denote by
$(H^2(D_\rho))'$ the Banach dual of $H^2(D_\rho)$, that is, the Banach
space of bounded linear functionals $\ell \colon H^2(D_\rho) \rightarrow \mathbb{C}$,
then the Banach adjoint $\mathcal{K} \colon (H^2(D_\rho))' \rightarrow (H^2(D_\rho))'$
of the transfer operator $\mathcal{L} \colon H^2(D_\rho) \rightarrow H^2(D_\rho)$
is defined by $(\mathcal{K}\ell)(f)=\ell(\mathcal{L} f)$.
The Banach adjoint $\mathcal{K}$ and the transfer operator $\mathcal{L}$ have the same
spectra (including multiplicities of eigenvalues),
which together with the previous considerations yields the following
result:

\begin{corollary}\label{coro:2}
Assume the hypothesis of Theorem \ref{thrm:1}, $M\geq N^2 R^N$, and $\rho=\sqrt{rR/\gamma}$.
Denote by $\mathcal{K} \colon (H^2(D_\rho))'\rightarrow (H^2(D_\rho))'$ the (compact) Koopman operator
extended to the Banach dual of $H^2(D_\rho)$. Then there exist enumerations of the eigenvalues
of $\mathcal{K}$ and $(H_N^{(M)})^{\dagger} G_N^{(M)}$, $\lambda_k(\mathcal{K})$ and
$\lambda_k((H_N^{(M)})^{\dagger} G_N^{(M)})$ respectively, such that for each fixed $k$ there exist constants
$C>0$ and $0<b<1$ such that
\[
|\lambda_k((H_N^{(M)})^{\dagger} G_N^{(M)})-\lambda_k(\mathcal{K})|\leq C b^N
\qquad \text{ for } N \in \mathbb{N}\, .
\]
\end{corollary}

\begin{remark}
The Banach adjoint $\mathcal{K}$ is indeed an extension of
the `usual' Koopman operator to the dual space $(H^2(D_\rho))'$.
To see this, take a function $g\in L^2([-1,1])$ and introduce
the functional $\ell_g \colon H^2(D_\rho) \rightarrow \mathbb{C}$ by
\begin{equation}\label{eq:2.2.13}
\ell_g(f) = \int_{-1}^1 f(x) g(x) \, dx \, .
\end{equation}
Using the continuity of point evaluations on $H^2(D_\rho)$, that is,
$|f(x)|\leq \rho/\sqrt{\rho^2-x^2} \|f\|_{H^2(D_\rho)}$
(see the proof of Lemma \ref{lem:1}), one can easily show that
$\ell_g$ is bounded. Hence functionals of the type defined
in Equation (\ref{eq:2.2.13}) constitute a subspace of $(H^2(D_\rho))'$.
With Equation (\ref{eq:2.1.3a}) we obtain
\begin{equation*}
(\mathcal{K} \ell_g)(f)=\ell_g(\mathcal{L} f)=
\int_{-1}^1 (\mathcal{L} f)(x) g(x) dx = \int_{-1}^1 f(x) g(T(x)) \, dx
= \ell_{g\circ T}(f) \, .
\end{equation*}
Hence on the subspace of linear functionals of the form (\ref{eq:2.2.13}), $\mathcal{K}$ acts in the
`usual' way $g\mapsto g\circ T$.
\end{remark}

\section{Numerical illustration}\label{sec:3}

We compare our previous analytic estimates by applying EDMD to chaotic
one-dimensional expanding maps. For the set of
observables we take monomials $\psi_k(x)=x^k$, $k=0,1,\ldots,N-1$.

\subsection{Piecewise linear maps}\label{sec:3.1}

As our first model
we take a piecewise linear full branch map, the skewed doubling map in \eqref{eq:3.1} defined
on $[-1,1]$, repeated for convenience below 
\begin{equation*} 
T_D(x)= \left\{
\begin{array}{rcr}
-1+2 (x+1)/(1+a) & \mbox{ if } & -1 \leq x\leq a\\
1+2 (x-1)/(1-a) & \mbox{ if } & a < x \leq 1
\end{array} \right.,
\end{equation*}
where the parameter $a \in (-1,1)$ determines the skew of the map.
The eigenvalues of the transfer operator (\ref{eq:2.1.3}),
defined on a space of analytic functions, say $H^2(D_R)$ with
$R>1$, are simply given by $\lambda_n=((1+a)/2)^{n+1} + ((1-a)/2)^{n+1}$
with $n=0,1,\ldots$, see for example
\cite{MoSoOs_PTP81} for an elementary account.
The finite-rank approximation given by
the matrix $L_N$ (see Equation (\ref{eq:2.2.10})) becomes exact in this piecewise linear case
and the first $N$ eigenvalues are given by the diagonal matrix
elements $(L_N)_{kk}$. Furthermore, if we perform EDMD with an infinite number
of nodes along the lines of Equation (\ref{eq:2.2.2}), EDMD is equivalent to the eigenvalue
problem of the matrix $L_N$ (see the proof of Lemma \ref{lem:10}, since $(1-P_N)L P_N=0$
in the case of piecewise linear maps). This is also reflected
in the numerical evaluation of EDMD where results for the leading $N$
eigenvalues coincide with the exact values to the numerical accuracy used, when
the matrices $H_N$ and $G_N$ are computed via Equation (\ref{eq:2.2.2}).
We note that computing the eigenvalues via EDMD requires solving the
generalised eigenvalue problem (\ref{eq:1.1}), which in itself poses a considerable
numerical challenge. For our numerics we have relied on standard commercial numerical solvers
contained, for example, in the Maple or Mathematica software packages. Alternatively, one can
use standard numerical eigenvalue solvers by employing the $\varepsilon$-pseudoinverse
of the matrix $H_N$, which, in practice, yields the same numerical
result for the eigenvalues.

The picture differs if we focus on EDMD for a finite number $M$ of equidistant
nodes $x_m=-1+(2m-1)/M$, $m=1,\ldots,M$, see the right panel of Figure~\ref{fig:1.1}. Then the collocation error
in computing the integrals, see Proposition \ref{lem:7}, causes corresponding errors
for the eigenvalues themselves. If we denote by
$\Delta_n=|\lambda_n-\lambda_n^{(EDMD)}|$ the absolute difference between the
exact eigenvalue and the corresponding result produced by EDMD, then
this error shows indeed the characteristic $1/M$ dependence if the leading
eigenvalues are considered, see Figure \ref{fig:2.1}. In addition, the error
shows large fluctuations which may be the result of the constant
expansivity in the piecewise linear chaotic map. The number of observables had
no impact on the error when infinite number of nodes were considered.
For a finite number of nodes a fairly small
number of observables is sufficient to obtain small errors and the error
may in fact increase (as seen in the right panel of Figure~\ref{fig:1.1}) if the number of observables is increased, see Theorem \ref{thrm:1}.
We will elaborate further on this observation in the next section.

\begin{figure}[!h]
\includegraphics[width=0.5 \textwidth]{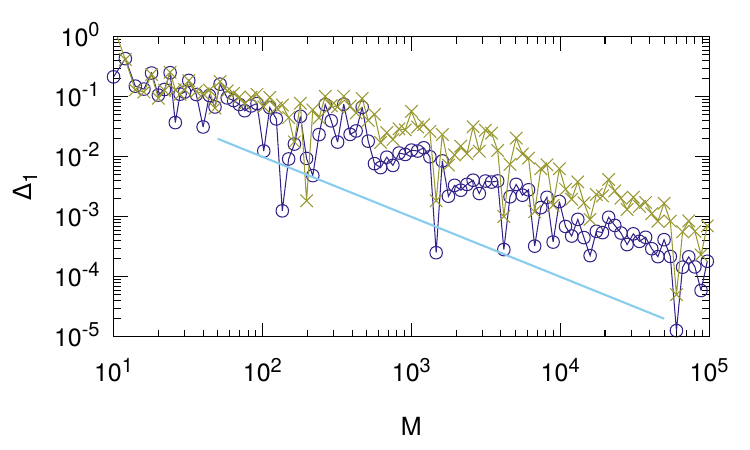}
\includegraphics[width=0.5 \textwidth]{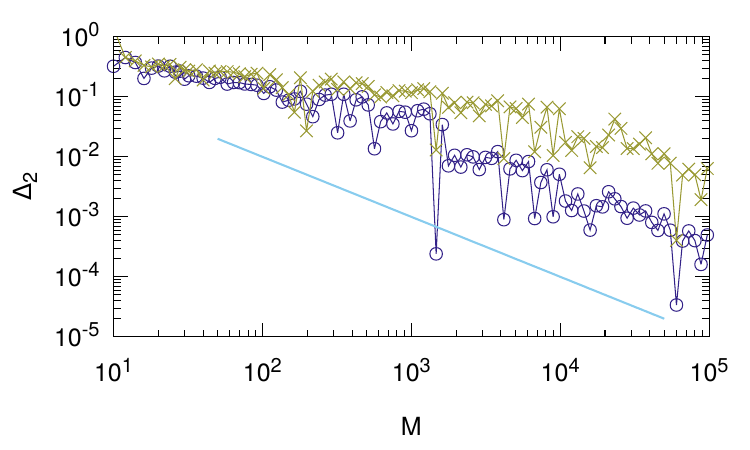}
\caption{Difference $\Delta_n$ between the exact
eigenvalue of the transfer operator
and the approximate eigenvalue obtained by EDMD
using monomials as observables, for the
skewed doubling map, Equation (\ref{eq:3.1}), with $a=1/\sqrt{2}$, in dependence on the number
of equidistant nodes $M$ on a double logarithmic scale.
Left: Error of the subleading eigenvalue, $\Delta_1$, for
$N=5$ (blue, circles) and $N=6$ (amber, cross). The line (cyan) shows
an algebraic decay of order $1/M$.
Left: Error of the second nontrivial eigenvalue, $\Delta_2$, for
$N=5$ (blue, circles) and $N=6$ (amber, cross). The line (cyan) shows
an algebraic decay of order $1/M$.}
\label{fig:2.1}
\end{figure}

\subsection{Blaschke maps}\label{sec:3.2}

To obtain further insight let us consider a piecewise analytic map with
nonlinear branches. For a reliable error estimate we focus
on the class of so-called interval Blaschke maps, for which the exact spectrum
of the transfer operator on $H^2(D_R)$ can been computed explicitly~\cite{SlBaJu_NONL13,BJS_AHP}.
We consider here the simplest case which is given by
a symmetric nonlinear deformation of the doubling map
\begin{equation}\label{eq:3.2}
T_B(x)= \left\{
\begin{array}{rcr}
\displaystyle 2x + 1  + \frac{2}{\pi} \mbox{arctan}\left(
\frac{\mu \sin(\pi x)}{1-\mu \cos(\pi x)}\right)
& \mbox{ if } & -1 \leq x\leq 0\\
\displaystyle 2x - 1  + \frac{2}{\pi} \mbox{arctan}\left(
\frac{\mu \sin(\pi x)}{1-\mu \cos(\pi x)}\right)
 & \mbox{ if } & 0 < x \leq 1
\end{array} \right.
\end{equation}
where $\mu\in (-1,1)$ denotes the parameter of the map.
The exact spectrum consists of the trivial eigenvalue
$\lambda_0=1$ and two nontrivial families of eigenvalues given by
$\{\mu^n : n=1,2,\ldots \}$ and $\{ (\mu/2+1/2)^n : n=1,2,\ldots \}$.
Eigenvalues in the first family have geometric multiplicity two, while 
eigenvalues in the second family are simple. The two inverse branches
can be computed in a straightforward way and they are given by
$\varphi_{\ell}(x)=x/2 +(-1)^\ell \mbox{arccos}(\mu \cos(\pi x/2))/\pi$
with $\ell=1,2$. If $|\mu|\leq 0.3$ then there exists some $R>1$ so that
these branches are analytic on a disk $D_R$ in the
complex plane and thus satisfy the conditions required for the
transfer operator to be compact, see Section \ref{sec:2.1} for details.

We first focus on EDMD with an infinite number of nodes, see Equation (\ref{eq:2.2.2}). In this case
(see Proposition \ref{lem:4}) the finite-rank approximation by the matrix $L_N$, Equation (\ref{eq:2.2.10}),
produces errors which are exponentially small in $N$.
EDMD yields eigenvalue estimates which are exponentially close to the
exact values (Proposition \ref{lem:11}),
and the errors $\Delta_n$ decay exponentially with the number of
observables used, as can also be observed numerically in
Figure \ref{fig:2.2}. We note that the two different sets of eigenvalues
which are contained in the spectrum of the transfer operator have different
exponential decay rates, which might be caused by the multiplicities of
the two sets of eigenvalues.

\begin{figure}[!h]
\centering \includegraphics[width=0.7 \textwidth]{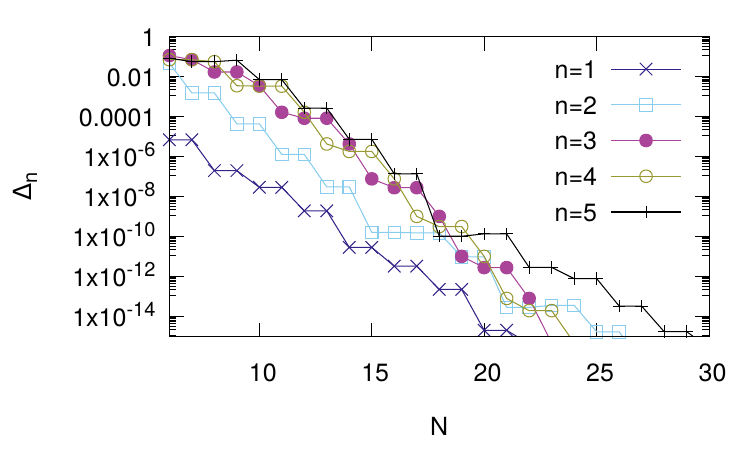}
\caption{Difference $\Delta_n$ between the exact
eigenvalue of the transfer operator
and the approximate eigenvalue obtained by EDMD
using monomials as observables and an infinite number of nodes, for a
Blaschke interval map, Equation (\ref{eq:3.2}), with $\mu=0.3$ in dependence on the
number of observables $N$ on a semi-logarithmic scale.
Results for the first five
nontrivial eigenvalues are shown. EDMD has been evaluated
for an infinite number of nodes, see Equation (\ref{eq:2.2.2}).
}
\label{fig:2.2}
\end{figure}

As a next step we investigate the impact of a finite number of nodes,
when EDMD is applied to the Blaschke map (\ref{eq:3.2}) with a
fixed number $N$ of observables. We again choose $M$ centralised nodes
$x_m=-1+(2m-1)/M$, $m=1,\ldots,M$ to evaluate the sums in Equations (\ref{eq:2.2.1}).
While the finite number of observables causes only an exponentially
small error, the total error is now dominated by the collocation errors
of the matrix elements (Proposition \ref{lem:7}), which are algebraic in the
number of nodes used. Since we are in a symmetric case with nodes
being interval midpoints, we expect the collocation errors to be of order $\mathcal{O}(M^{-2})$,
as indeed observed in Figure \ref{fig:2.3} for the leading eigenvalues.
However, the asymptotic behaviour sets in later
for larger number of nodes 
when higher eigenvalues are considered (see also Theorem \ref{thrm:1}).

\begin{figure}[!h]
\centering \includegraphics[width=0.7 \textwidth]{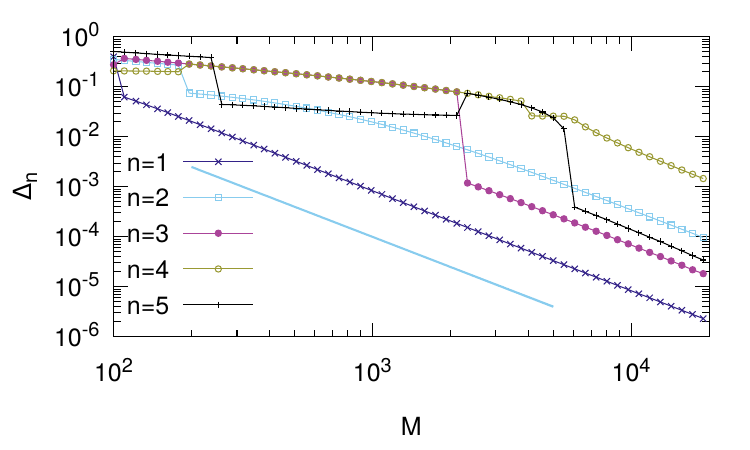}
\caption{Difference between the exact
eigenvalue of the transfer operator
and the approximate eigenvalue obtained by EDMD
using monomials as observables, for a
Blaschke interval map, Equation (\ref{eq:3.2}), with $\mu=0.3$ in dependence on the
number of equidistant nodes $M$ on a double logarithmic scale.
Results for the first five nontrivial eigenvalues are shown.
The jumps are likely to be
artefacts caused by sorting eigenvalues by size of modulus.
EDMD has been evaluated for $N=15$ observables.
The line (cyan) shows an algebraic decay of order $1/M^2$.
}
\label{fig:2.3}
\end{figure}

We will now demonstrate the joint dependence of the eigenvalue approximation error on both
the number of observables $N$ and the number of nodes
$M$, again using the Blaschke map (\ref{eq:3.2}), with observables
given by monomials, and a set of centralised equidistant nodes.
As already observed in
Figure \ref{fig:2.1}, we can expect to require a sufficiently large number of
nodes to deal with expressions containing monomials of high power, so that convergence
of eigenvalues depends on a nontrivial relation between $N$ and $M$ (see Corollary \ref{coro:1}).
The results
 displayed in Figure \ref{fig:2.4} confirm this expected lack of uniformity as convergence
seems to occur in a triangular shaped region when $N$ and $M$ jointly tend to infinity.
Numerically, a constraint of the type $M > c N^2$ seems to be required to ensure
convergence of the EDMD eigenvalues towards the correct spectrum of the
transfer operator. Hence, the numerical result seems to indicate that the terms
$\gamma^N$ occurring in Theorem \ref{thrm:1} overestimate the actual error,
and a tighter bound avoiding such terms might be obtained.
In fact, the statement of Theorem \ref{thrm:1} is uniform across the spectrum,
and the estimate might be improved by
using $\varepsilon$-pseudoinverses for operators and focusing solely on the
leading part of the spectrum.

\begin{figure}[h!]
\includegraphics[width=0.5 \textwidth]{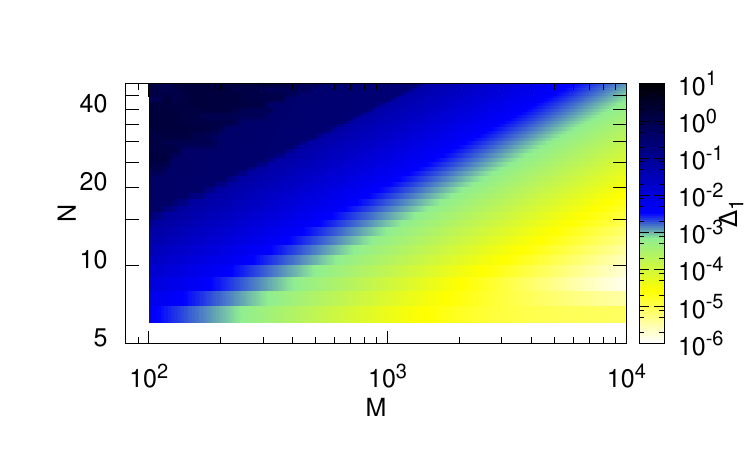}
\includegraphics[width=0.5 \textwidth]{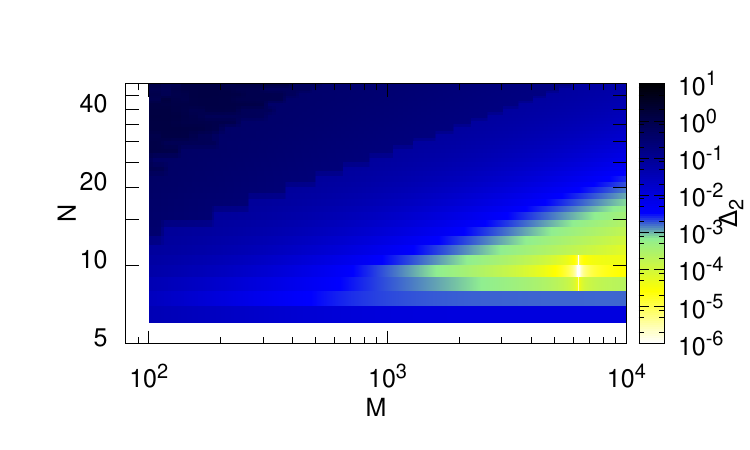}
\caption{Density plot of the difference between the exact
eigenvalue of the transfer operator
and the approximate eigenvalue obtained by EDMD
using monomials as observables, for a
Blaschke interval map with $\mu=0.3$ in dependence on the
number of equidistant nodes $M$  and the number of observables $N$.
The data are shown on a double logarithmic scale.
Left: error of the subleading eigenvalue, $\Delta_1$, right: error of the
second nontrivial eigenvalue, $\Delta_2$.
}
\label{fig:2.4}
\end{figure}

In line with Theorem \ref{thrm:1}, the numerical investigation of simple toy models
shows exponential convergence in the number of
observables, so that a small number of observables turns out
to be sufficient. However the required number of nodes $M$ is large
and grows rapidly with increasing $N$, and we observe a nontrivial relation
between $N$ and $M$ which needs to be satisfied to obtain convergence of the EDMD
algorithm.

\subsection{Impact of observables}\label{sec:3.3}

We have derived our results for simple toy models where a rigorous
account for our findings can be provided. In general, however, the quality of the
EDMD results is hard to judge, as the outcome depends on the observables in
a subtle way. Even when EDMD converges, the meaning of the
eigenvalues obtained may not be obvious. We have demonstrated this phenomenon
for the skewed doubling map when EDMD is applied with observation through
Fourier modes, see the left panel of Figure~\ref{fig:1.1}.

Let us look at this case in some more detail. To simplify the analysis
consider the case of infinite number of nodes where the relevant matrix
elements are simply given by
\begin{align}\label{eq:4.1}
H_{k\ell} &= \frac{1}{2} \int_{-1}^1 \exp(i \pi k x) \exp(-i \pi
\ell x) \, dx = \delta_{k \ell}\nonumber \\
G_{k \ell} &= \frac{1}{2} \int_{-1}^{1} \exp(i \pi k T(x))
\exp(-i \pi \ell x) \, dx \nonumber \\
&= \frac{1+a}{2}  e^{i \pi \ell(1-a)/2}
\frac{\sin(\pi(k-\ell(1+a)/2))}{\pi(k-\ell(1+a)/2)} \nonumber \\
&  \qquad + \frac{1-a}{2}  e^{-i \pi \ell(1+a)/2}
\frac{\sin(\pi(k-\ell(1-a)/2))}{\pi(k-\ell(1-a)/2)}
\end{align}
when using Equation (\ref{eq:3.1}). While not crucial,
we have used here the complex conjugate for the second entry
as that conveniently results for $H$
in  the identity matrix. 
Figure \ref{fig:2.5} shows the result for the
subleading eigenvalue $\lambda_1$ computed from the generalised eigenvalue problem,
that is, the subleading eigenvalue of $G_N$, in dependence on the skew of the map.
We obtain convergence when $N$ increases, but convergence slows down
considerably for maps
with a small skew $a$. Furthermore, the modulus of the subleading eigenvalue
seems to converge towards $(1+|a|)/2$,
an expression
which coincides with the essential spectral radius of the transfer
operator considered
on the space of functions of bounded variation, see for example~\cite{HoKe_MZ82}.
With hindsight we can also provide an analytic argument supporting this observation.
We may consider the skewed doubling map as a piecewise analytic map on the complex
unit circle, following the construction which is used for analytic maps, see for example~\cite{SlBaJu_NONL13}.
Fourier modes become complex monomials, and smooth functions on the interval become discontinuous
on the unit circle, as smooth functions on an interval normally do not have a continuous
periodic extension beyond the endpoints of the interval. If one applies the transfer
operator to such discontinuous functions the number of discontinuities
may increase, but the variation of the function remains bounded because of the
expansivity of the map. While one would naively expect that Fourier modes
imply studying transfer operators on a space of square integrable functions we
encounter here an unexpected subtle mechanism where Fourier modes
result in studying the transfer operator on the space of functions with bounded
variation. This observations parallels the approximation of
smooth maps by piecewise linear Markov maps, see for example \cite{SoYoOkMo_JSP84,JuFu_PD93}, which
effectively also amounts to studying transfer operators on the space of functions
with bounded variation \cite{SlBaJu_JPA13}.
In summary, it is not surprising that the corresponding essential spectral radius
shows up in the numerical
outcome of EDMD. It is also worth mentioning that the essential spectral radius
does not have to be an eigenvalue of the transfer operator, so that some care
has to be taken when interpreting eigenmodes obtained by EDMD.

\begin{figure}[!h]
\centering
\includegraphics[width=0.7 \textwidth]{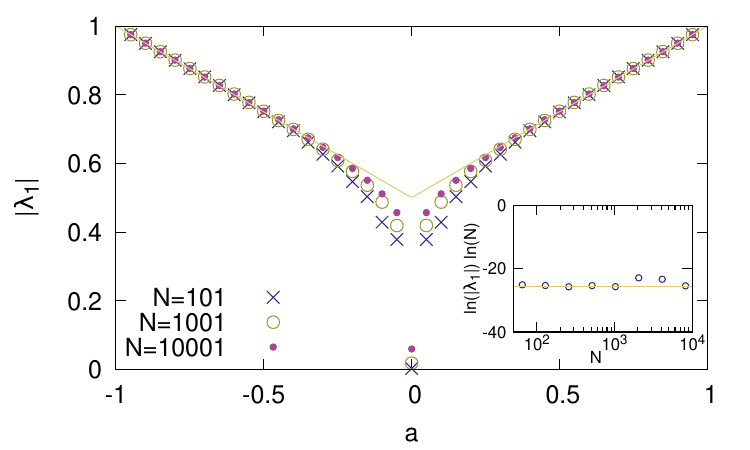}

\caption{Absolute value of the subleading eigenvalue
$\lambda_1$ obtained for the
skewed doubling map, Equation (\ref{eq:3.1}), in dependence on the skew $a$,
using EDMD with $N$ Fourier modes in the limit
of infinitely many nodes, see Equations (\ref{eq:4.1}) (symbols). The line (amber) shows the
essential spectral radius of the transfer operator defined on the
space of functions with bounded variation. The inset shows the
product $\ln|\lambda_1| \ln N$ as a function of $N$ for a doubling map
with tiny skew $a=10^{-16}$ (symbols). The horizontal line (amber) indicates
a crude analytic estimate (see the text for details).
}
\label{fig:2.5}
\end{figure}

Let us also comment on the slow convergence which is visible in Figure
\ref{fig:2.5} close to $a=0$. The transfer operator of the doubling map, $a=0$,
effectively halves wavenumbers of the Fourier modes, that is, the matrix $G_N$ in Equation (\ref{eq:4.1})
has a single eigenvalue $1$ and a degenerate eigenvalue zero. The matrix $G_N$ for
$a=0$ consists of Jordan blocks, and the largest block has size $b$ where
$2^b \sim N$. For $a \neq 0$, standard perturbation theory tells us
that nonvanishing eigenvalues of the size $|\lambda_1| \sim |a|^{1/b} \sim |a|^{\ln 2/\ln N}$
are generated. This simple reasoning is consistent with the numerical data,
see the inset in Figure \ref{fig:2.5}. It also supports the assertion that EDMD in this
example yields the essential spectral radius for the transfer operator being defined
on the space of functions with bounded variation.

In summary, while EDMD is a powerful data analysis tool, results
have to be taken with a grain of salt, as the interpretation of the eigenvalues and
eigenmodes is far from obvious. In particular, it is often not clear which transfer or Koopman operator
is studied by EDMD, as the underlying operator (together with a function space) is determined in a subtle way by the observables used, if it exists at all. Nevertheless, at the phenomenological level EDMD is a very appealing tool, and further rigorous studies are required to complement
its empirical success with provable convergence guarantees and rates. 

\section*{Acknowledgement} O.F.B. gratefully acknowledges support by EPSRC though grant
EP/R012008/1, J.S. acknowledges support by ERC-Advanced Grant 833802-Resonances, and W.J.
acknowledges funding by German Research Foundation SFB 1270/2 - 299150580.

\paragraph{Declaration of interests:} None

\end{document}